\documentclass[preprint,10pt]{amsart}

\usepackage{amssymb}
\usepackage{mathrsfs}   

\usepackage{tikz}   

\usepackage[hang,small,bf]{caption}
\theoremstyle{plain}
\newtheorem{theorem}{Theorem}[section]
\newtheorem{corollary}[theorem]{Corollary}
\newtheorem{lemma}[theorem]{Lemma}

\theoremstyle{definition}

\theoremstyle{remark}

\newcommand\mL{L\kern-0.08cm\char39}  

\begin{document}

\title{Transitivity in nonautonomous systems}


\author[M. Ml\'{\i}chov\'a, V. Pravec]{Michaela Ml\'{\i}chov\'a, Vojt\v ech Pravec}
\email{Michaela.Mlichova@math.slu.cz, Vojtech.Pravec@math.slu.cz}

\address{Mathematical Institute, Silesian University, 746 01 Opava, Czech Republic}

\maketitle

\begin{abstract}
Let $(X,d)$ be a metric space and $f_{1,\infty}=\{f_n\}_{i=0}^{\infty}$  be a sequence of continuous maps $f_n:X\rightarrow X$ such that $(f_n)$ converges uniformly to a continuous map $f$. We investigate which conditions ensure that the transitivity of functions $f_n$ or the transitivity of the nonautonomous system $(X,f_{1,\infty})$ is inherited to the limit function $f$ and vice versa. Such problem have been studied for instance by A. Fedeli, A. Le Donne or J. Li who give different sufficient condition for inheriting of transitivity from $f_n$ to $f$. In this paper we give a survey of know result relating to this problem and prove new results concerning to transitivity.
\end{abstract}

%

\section{Introduction}\label{section1}

Throughout the paper let $(X, d)$ be a compact metric space unless it is emphasized differently and $I=[0,1]$ be the compact unit interval. Denote by $C(X)$, resp. $C(I)$, the set of continuous maps $X\rightarrow X$, resp. $I\rightarrow I$. Let $f_{1,\infty}=\{f_n\}_{n=1}^{\infty}$ be a sequence of continuous map $f_n\in C(X)$. By \textit{nonautonomous dynamical system} (NDS, for short) we mean a pair $(X,f_{1,\infty})$. For $(X,f_{1,\infty})$ and any $n, k \in \mathbb N$ we denote
$$f_n^0=id_X,\quad f_n^k:= f_{n+k-1} \circ \ldots \circ f_{n+1}\circ f_n,$$ i.e., the $k$-th iteration of $f_{1,\infty}$ starting from $f_n$
and
$$(f_n)^0=id_X,\quad(f_n)^k:= \underbrace{ f_{n} \circ \ldots \circ f_{n}\circ f_n}_{k\,times},$$ i.e., the $k$-th iteration of fiber map $f_i$.\\
\indent Recall that the system $(X,f_{1,\infty})$ is \textit{topologically transitive} if for any two open sets $U,V$ there is $n\in\mathbb{N}$ such that $f_1^n(U)\cap V\neq \emptyset$ and the system $(X,f_{1,\infty})$ has a \textit{dense orbit} if there is a point $x\in X$ whose orbit $\mathcal{O}(x,f_{1,\infty})=\{f_1^n(x);n\in\mathbb{N}_0\}$ is dense in $X$. Also note that an autonomous dynamical system $(X,f)$ is a special of NDS with $f_n=f$ for any natural number $n$.

We assume that $f_n$ is a surjective map for any $n\in\mathbb{N}$. Note that such assumption is very natural since we want to avoid some pathological examples (constant function etc.). Also we focus only on uniformly convergent NDS, i.e. such that the sequence $f_{1,\infty}$ converges uniformly to a continuous map $f$. 

Very important and frequently studied question concerning NDS is what properties are inherited from NDS $(X,f_{1,\infty})$, resp. from fiber functions $f_n$, to a limit function $f$ or vice versa. Note that in recent years such question has been intensively studied considering different properties (see e.g. \cite{C1}, \cite{C2}, \cite{Dv}, \cite{St}).\\
 In this paper we focus on a very important property which is topological transitivity and the existence dense orbit of dynamical system. It is known that even for interval NDS the assumption of uniform convergence is insufficient to ensure the inheritance of any of these conditions. More precisely one can construct a uniformly convergent NDS $f_{1,\infty}$ such that each $f_n$ is topological transitive and with dense orbit but the sequence $f_{1,\infty}$ converges to identity. On the other hand there is a interval NDS $f_{1,\infty}$ such that each $f_n$ is not topological transitive and without dense orbit but the sequence $f_{1,\infty}$ converges to the tent map. Therefore natural question is what condition a NDS or the fiber functions have to satisfy to ensure the inheritance of topological transitivity or dense orbit to limit function $f$ or vice versa. 

In recent years several authors have studied this problem and they have introduced various conditions that at least partially help to answer the question. In the following subsection we give a survey of them and in Section 2 we summarize known and new results.
\subsection{Survey of conditions}

Note that some of following conditions were labeled differently by their authors. We change the labeling for the sake of simplicity and uniformity. Conditions concerning iterations of fiber function are labeled with a star, i.e., *.
Let $D$ be the supremum metric on $C(X)$, i.e., $D(f,g)=\sup_{x\in X}d(f(x),g(x))$ then we say that the system satisfies condition:
\begin{enumerate}
\item[(CC)] if for every $\varepsilon>0$ there is an $n_0 \in \mathbb N$ such that for every $n \ge n_0$,  $k \in \mathbb N$ and every $x \in X$
	$$ d(f^k_n(x), f^k(x))< \varepsilon,$$
	
\item[(CC*)]  if for every $\varepsilon>0$ there is an $n_0 \in \mathbb N$ such that for every $n \ge n_0$,  $k \in \mathbb N$ and every $x \in X$
	$$ d((f_n)^k(x), f^k(x))< \varepsilon,$$
	
\item[(L)] if $$\lim_{n\to \infty} D(f^n_n, f^n)=0,$$

\item[(L*)] if $$\lim_{n\to \infty} D((f_n)^n, f^n)=0,$$

\item[(DO)] if there is an $x_0 \in X$ such that $\{ f_n^n (x_0) \}$ is dense in $X$,

\item[(DO*)] if there is an $x_0 \in X$ such that $\{ (f_n)^n (x_0) \}$ is dense in $X$

\end{enumerate}

\bigskip
Conditions (L*) and (DO*) are studied in \cite{Li}, \cite{RF} and \cite{YZZ}, condition (CC*) is studied in \cite{FD} and \cite{Li} and condition is studied in \cite{SR}. Conditions (DO) and (L) are introduced as natural counterpart of conditions (DO*) and (L*).\\
In the following text the systems to which the conditions (CC*), (L*) or (DO*) can be applied are denoted by FDS - fiber dynamical system, whereas the conditions (CC), (L) or (DO) can be applied to NDS.

\section{Transitivity}\label{trans}


The following Theorem was proved by Yan, Zeng and Zahng in \cite{YZZ} and it shows that if a FDS  satisfies the condition (L*) then the property (DO*) is equivalent with the topological transitivity of limit function $f$.

\begin{theorem} \cite{YZZ}
Let $(X,d)$ be a compact metric space with no isolated point, $f_n\in C(X)$, $n\in\mathbb{N}$ such that $f$ is uniform limit of $f_{1,\infty}$. Let $f_{1,\infty}$ satisfies the condition $(L*)$. Then the following are equivalent:
\begin{enumerate}
\item for any two nonempty open subsets $U$ and $V$ of $X$, $U\cap (f_n)^{-n}(V)\neq \emptyset$ for some $n\in\mathbb{N}$;
\item $\{x\in X:\{(f_n)^n(x)\}\mbox{ is dense in }X\}$ is a dense $G_{\delta} set$;
\item $f_{1,\infty}$ satisfies condition (DO*);
\item $f$ is topologically transitive.
\end{enumerate} 
\end{theorem}
An analogous result is true also for nonautonomous dynamical system and the proof is inspired by the proof of the original Theorem. Note that in the following Theorem we define $f_n^{-n}$ as follows:
$$f_n^{-n}=(f_n^{n})^{-1}=f_n^{-1}\circ f_{n+1}^{-1}\circ\dots\circ f_{2n-2}^{-1}\circ f_{2n-1}^{-1}$$
\begin{theorem}
Let $(X,d)$ be a compact metric space with no isolated point, $f_n\in C(X)$, $n\in\mathbb{N}$ such that $f$ is uniform limit of $f_{1,\infty}$. Let $f_{1,\infty}$ satisfies the condition $(L)$. Then the following are equivalent:
\begin{enumerate}
\item for any two nonempty open subsets $U$ and $V$ of $X$, $U\cap f_n^{-n}(V)\neq \emptyset$ for some $n\in\mathbb{N}$;
\item $\{x\in X:\{f_n^n(x)\}\mbox{ is dense in }X\}$ is a dense $G_{\delta} set$;
\item $f_{1,\infty}$ satisfies condition (DO);
\item $f$ is topologically transitive.
\end{enumerate} 
\end{theorem}
\begin{proof}
$(1)\implies (2)$: Let $\{U_i\}_{i=1}^{\infty}$ be a countable base of $X$ and let $A=\{x:\overline{\{f_n^n(x)\}}=X\}$. Then it is easy to see that $x\in A$ if and only if for any $n\in\mathbb{N}$ $x\in \cup_{m=1}^{\infty}f_m^{-m}(U_n)$ which is equivalent to 
$$x\in\bigcap_{n=1}^{\infty}\bigcup_{m=1}^{\infty}f_m^{-m}(U_n)$$ therefore
$$\{x:\overline{\{f_n^n(x)\}}=X\}=\bigcap_{n=1}^{\infty}\bigcup_{m=1}^{\infty}f_m^{-m}(U_n)$$
and the result follows.

$(2)\implies (3)$ Obvious.

$(3)\implies (4)$ Since system $f_{1,\infty}$ satisfies condition (DO) there is $x\in X$ such that $\{f_n^n(x)\}$ is dense in $X$. Fix $y\in X$ and $\varepsilon>0$ then there is $n_0\in\mathbb{N}$ such that $D(f_n^n,f^n)<\varepsilon/2$ for any $n\geq n_0$. Since $X$ is a metric space without isolated point the set $B=B_{\varepsilon}(x)\setminus\{x,f_1^1(x),f_2^2(x),\dots,f_{n_0}^{n_0}(x)\}$ is a nonempty open set, where $B_{\varepsilon}(x)=\{z\in X;d(x,z)<\varepsilon\}$, there is $n_1\geq n_0$ such that $d(f_{n_1}^{n_1}(x),y)<\varepsilon/2$. We get
$$d(f^n(x),y)\leq d(f^n(x),f_n^n(x))+d(f^n_n(x),y)<\varepsilon,$$ 
which shows that $\{f^n(x)\}$ is dense in $X$.\\
Now let $U,V$ be arbitrary nonempty open subset of $X$, then there is $p\in\mathbb{N}$ such that $f^p(x)\in U$. Again from the fact that $X$ is without isolate point we get that the set $W=V\setminus\{x,f(x),f^2(x),\dots,f^p(x)\}$ is a nonempty open set of $X$. Since $\{f^n(x)\}$ is dense in $X$ there is $q>p$ such that $f^q(x)\in W\subset V$, that implies $f^q(x)=f^{q-p}(f^p(x))\in f^{q-p}(U)\cap V$. Hence $f^{q-p}(U)\cap V$ is nonempty and $f$ is topologically transitive.

$(4)\implies (1)$ Let $U,V$ be arbitrary nonempty open subset of $X$. Fix $v\in V$ and $\varepsilon>0$ such that $B_{\varepsilon}(v)\subset V$. Since $f$ is topologically transitive and the system $f_{1,\infty}$ satisfy condition (L) there is $n\in\mathbb{N}$ such that $U\cap f^{-n}(B_{\varepsilon/2}(v))\neq \emptyset$ and $D(f_n^n,f^n)<\varepsilon/2$. Choose $u\in U\cap f^{-n}(B_{\varepsilon/2}(v))$.
$$d(f_n^n(u),v)\leq d(f_n^n(u),f^n(u)+d(f^n(u),v)<\varepsilon$$
Thus $f_n^n(u)\in B_{\varepsilon}(v)\subset V$ and therefore $u\in U\cap f_n^{-n}(V)$ which completes the proof.

\end{proof}
It is well known (see \cite{KS}) that in the case of autonomous dynamical system $(X,f)$ the topological transitivity and existence of dense orbit are mutually equivalent if $X$ is a compact space without isolated point. But even for interval NDS such property is no longer true, in general interval NDS dense orbit does not imply topological transitivity (for example if for any $n\geq 2,\,f_n=f $ where $f$ is a transitive interval map with fixed point $c$ and $f_1|_J=c$, where $J$ is a nontrivial subinterval of $I$). The following lemma shows that even if the system $(X,f_{1,\infty})$ satisfies conditions (L) and (L*) then there is no relation between topological transitivity of fiber maps and the condition (DO) and (DO*).

\begin{lemma}
Let $\mathbb S$ be a circle (with radius $1/2\pi$). There exist $f_n, g_n, f, g \in C(\mathbb S)$ such that $f$ is uniform limit of $f_{1,\infty}$ and $g$ is uniform limit of $g_{1,\infty}$ ,  the condition (L*), (L) are fulfilled and
\begin{enumerate}
\item $f_n$ is transitive, (DO*) and (DO) are not satisfying,
\item $g_n$ is not transitive, (DO*) and (DO) is satisfying.
\end{enumerate}
For both cases, the condition (CC*) is not satisfying and (CC) is satisfying.
\end{lemma}

\begin{proof} For any $n \in \mathbb N$, let $f_n: \mathbb S \rightarrow \mathbb S$ be a rotation with the angle $\frac{\pi}{2^n}$. Obviously, $f_n \rightrightarrows Id=:f$ and 
$$
d((f_n)^n, f^n)=d((f_n)^n, Id)= \frac{n}{2^n},
$$
$$
d(f_n^n, f^n)=d(f_n^n, Id)= \frac{1}{2^n}+\frac{1}{2^{n+1}}+ \cdots + \frac{1}{2^{2n-1}} < \frac{n}{2^n}.
$$
Then the condition (L*) and (L) are satisfying. For any $x \in \mathbb S$ and $n \in \mathbb N$, $(f_n)^n(x)$ is contained in the same half circle, thus (DO*) is not satisfying whereas $f_n$ is transitive for every $n \in \mathbb N$. 

Let $g: \mathbb S \rightarrow \mathbb S$ be a irrational rotation and $g_n$ be a rational rotation such that $g_n \rightrightarrows f$ and (L*). Since $g$ is transitive, by \cite{RF}, (DO*) is satisfying whereas $g_n$ is not transitive for any $n \in \mathbb N$.
\end{proof}

\bigskip

 Fedeli and Le Donne in \cite{FD} gave the following condition sufficient for transitivity of the limit function. They called the property (CC*)  {\it orbital convergence}.
 
 \begin{theorem}{\label{FD}}
\cite{FD}
 {Let (X,d) be a metric space (need not be neither perfect nor compact), $f_n \in C(X)$, $n\in \mathbb{N}$ such that $f$ is uniform limit of $f_{1,\infty}$, let $f_n$ be topologically transitive and satisfy (CC*). Then $f$ is topologically transitive.}
\end{theorem}
 
 Sharma and Raghav in \cite{SR} {
  showed the followig connection between topological transitivity of NDS and its uniform limit:
 
 \begin{theorem}
 {\cite{SR}}
 Let $X$ be a compact metric space, all functions are supposed to be surjective. Let $f_n \in C(X)$ are feebly open $\forall n\in\mathbb N$, $f$ is uniform limit of $f_{1,\infty}$ and the conditon (CC) is fulfilled. Then $f$ is transitive if and only if $f_{1,\infty}$ is transitive. 
 \end{theorem}
 The following theorem is a slight generalization of Theorem \ref{FD}. We omit the proof since it is analogous to the proof of original Theorem.
  \begin{theorem}\label{CC* upravene}
 Let $X$ be a compact metric space, all functions are supposed to be surjective. Let $f_n \in C(X)$, $n\in \mathbb{N}$, $f$ is uniform limit of $f_{1,\infty}$, the conditon (CC*) is fulfilled, and $\exists n_0$ such that,  $\forall n\geq n_0$, $f_n$ is transitive.  Then $(X, f)$ is transitive.
 \end{theorem}
 
As is showed in the following Lemma, Theorem \ref{CC* upravene} can not be improved to the equivalence form.
\begin{lemma}  Let $X=\{0,1\}^{\mathbb N}$ be the Cantor set equipped with metric \\
$\rho(x_1x_2\cdots, y_1y_2\cdots)= 1/n$ where $n$ is the minimal index such that $x_n\ne y_n$. Let $f\in\mathcal C(X)$ be the standard adding machine {\rm mod} 2 (i.e. $f(x_1x_2x_3\dots)=((x_1x_2x_3\dots)+(100\dots))\mod 2$ and the addition of 1 is carried to next coordinate if $x_i+1=2$) hence a transitive map. Then there is a nonautonomous system $f_{1,\infty}\in \mathcal C(X)$ converging uniformly to $f$ such that every point of $f_n$ is periodic and
\begin{equation}
\label{12}
\rho((f_n)^k(x),f^k(x) )\le \frac 1{n+1}, \ k, n\in\mathbb N,\, x\in X.
\end{equation} 
Consequently, on $X$, condition $(CC^*)$ and uniform convergence of a non-autonomous system to a transitive map does not imply transitivity of maps in the nonautonomous system.\\
\end{lemma}
\begin{proof} It suffices to let $f_n (x_1x_2\cdots x_nx_{n+1}\cdots) =x^*_1x^*_2\cdots x^*_nx_{n+1}x_{n+2}\cdots)$ where $x^*_1x^*_2\cdots x^*_n$ 
denotes  the first $n$ symbols of $f(x_1x_2\cdots x_nx_{n+1}\cdots)$. So, $f_n$ acts on the first $n$ symbols as the adding machine, letting the remaining symbols unchanged.
\end{proof}
Recall that autonomous dynamical system $(X,f)$ has \textit{sensitive dependence on initial conditions} (is sensitive, for short) if there is $\delta>0$ such that for any $x\in X$ and $\varepsilon>0$ there is $y\in X$ with $d(x,y)<\varepsilon$ such that $d(f^n(x),f^n(y))>\delta$ for some $n\geq 0$.

Although, in general the converse implication of Theorem \ref{CC* upravene} is not true, for interval map we get the following result.
\begin{theorem}
Let f, $f_n \in C(I)$, $n\in \mathbb{N}$ be surjective maps such that f is uniform limit of $f_{1,\infty}$. If $f^2$ is transitive and $f_{1,\infty}$ satisfies condition (CC*), then for every $\varepsilon>0$ there is $N\in\mathbb{N}$ such that for any $n\geq N$ there is $J_n\supset (\varepsilon,1-\varepsilon)$ such that $f_n(J_n)=J_n$ and $f|_{J_n}$ is transitive.
\end{theorem}
\begin{proof}
First observe that since $f$ is transitive there is $\delta >0$ such that $f$ is $\delta$-sensitive. Then it is easy to see that from the condition (CC*) we get that there is $n_0\in\mathbb{N}$ and $\delta '>0$ such that for any $n\geq n_0$ $f_n$ is $\delta '$-sensitive. Fix $0<\varepsilon<1/4$, then there is $n_1\in\mathbb{N}$ such that for any $n\geq n_1$ 
\begin{equation}\label{CC*1}
D((f_n)^k,f^k)<\frac{\varepsilon}{2} 
\end{equation} 
Set $N=\max\{n_0,n_1\}$. 
Then for any $n\geq N$ there is a cycle of intervals $J_1^n,J_2^n,\dots,J_p^n$, i.e. $f_n(J_i^n)=J_{i+1}^n$ for $1\leq i< n$ and $f_n(J_p^n)=J_1^n$, such that $f_n|_{J_1^n\cup J_2^n\cup\dots\cup J_p^n}$ is transitive (see \cite{RU} for more details).

Fix $n\geq N$ and assume that $p>1$, then there is $i\leq p$ such that $|J_i^n|\leq 1/2$. Since $f$ is transitive there is $n_2\in\mathbb{N}$ such that for any $m\geq n_2$ $f^m(int\, J_i^n) \supseteq [\varepsilon/2,1-\varepsilon/2]$. For any $m=kp>n_2$ by (\ref{CC*1}) we get that 
$$[\varepsilon,1-\varepsilon]\subseteq (f_n)^m(int\,J_i^n)\subseteq J_i^n$$
which is a contradiction.
\end{proof}
Recall that if interval map $f$ is transitive but $f^2$ is not then there is a unique fixed point $u\in(0,1)$ such that $f([0,u])=[u,1]$ and $f([u,1])=[0,u]$. Moreover $f^2|_{[0,u]}$ and $f^2|_{[u,1]}$ are transitive. The proof of the following corollary of foregoing theorem is straightforward therefore we omit it.
\begin{corollary}

Let f, $f_n \in C(I)$, $n\in \mathbb{N}$ be surjective maps such that f is uniform limit of $f_{1,\infty}$. If $f$ is transitive with unique fixed point $u$ such that $f^2$ is not transitive and $f_{1,\infty}$ satisfies condition (CC*), then for every $\varepsilon>0$ there is $N\in\mathbb{N}$ such that for any $n\geq N$ there is $J_n\supset (\varepsilon,u-\varepsilon)$ and $K_n\supset (u+\varepsilon,1-\varepsilon)$ such that $f_n(J_n)=J_n,\,f_n(K_n)=K_n$ and $f|_{J_n},\,f|_{K_n}$ are transitive.
\end{corollary}
For piecewise linear maps of constant slope we get even stronger result.
\begin{theorem}\label{main1}
Let f, $f_n \in C(I)$, $n\in \mathbb{N}$ be surjective maps such that f is uniform limit of $f_{1,\infty}$. If f is transitive, piecewise linear map of constant slope $k$ and $f_{1,\infty}$ satisfies condition (CC*), then there is $n_0\in\mathbb{N}$ such that for every $n\geq n_0$ $f_n(x)=f(x)$.
\end{theorem}
We first show the following results which are used to prove the theorem. Note that in the following lemma the set of all fixed points of a function $f$ is denoted by $Fix(f)$.
\begin{lemma}\label{fix}
Let f, $f_n \in C(I)$, $n\in \mathbb{N}$ be surjective maps such that f is uniform limit of $f_{1,\infty}$. If f is transitive, piecewise monotone map, and $f_{1,\infty}$ satisfies condition (CC*), then there is $n_0\in \mathbb{N}$ such that for any $n\geq n_0$ $Fix(f)\subset Fix(f_n)$ and  $f^{-j}(Fix(f))=(f_n)^{-j}(Fix(f))$ for any natural number $j$.
\end{lemma}
\begin{proof}
First, since $f$ is piecewise monotone and transitive, it is easy to see that the set $Fix(f)$ is finite, and there is $\varepsilon>0$ such that for every $ x\in Fix(f)$ and for every $y\in (x-\varepsilon,x+\varepsilon)$ there is $l\in\mathbb{N}$ such that
\begin{equation}\label{okoli}
|f^l(x)-f^l(y)|=|x-f^l(y)|>3\varepsilon.
\end{equation}  
Now, there is $n_0\in \mathbb{N}$ such that for every $n\geq n_0$, every $k\in\mathbb{N}$ and every $x\in I$ 
\begin{equation}\label{CC*}
d((f_n)^k(x),f^k(x))<\varepsilon.
\end{equation}
Let $x$ be a fixed point of $f$, $n\geq n_0$ and assume that $f_n(x)\neq x$. Since $f_n(x)\in (x-\varepsilon, x+\varepsilon)$, from (\ref{okoli}) we get that there is integer $l$ such that 
$$|f^l(f_n(x))-f^l(f(x))|>3\varepsilon.$$
Using triangular inequality and (\ref{CC*}) we get 
$$|x-f^l(f_n(x))|\leq |x-(f_n)^l(f_n(x))|+|(f_n)^l(f_n(x))-f^l(f_n(x))|,\\$$
$$|(f_n)^l(f_n(x))-f^l(f_n(x))|\geq |x-f^l(f_n(x))|- |x-(f_n)^l(f_n(x))|>3\varepsilon-\varepsilon=2\varepsilon,$$
which is contradiction. Therefore $x\in Fix(f_n)$.\\
\\
Now let $y\in f^{-j}(x)$, $x\in Fix(f)$, $n\geq n_0$ and assume that $(f_n)^j(y)\neq x$. Again, from (\ref{okoli}) there is $l\in\mathbb{N}$ such that 
$$|f^l((f_n)^j(y))-f^l(f^j(y))|>3\varepsilon.$$
Using triangular inequality and (\ref{CC*}) we get 
$$|f^l(f^j(y))-f^l(f_n)^j(y))|\leq |f^l(f^j(y))-f^l_n((f_n)^j(y))|+|(f_n)^l(f_n)^j(y))-f^l(f_n)^j(y))|,\\$$
$$|(f_n)^l(f_n)^j(y)-f^l(f_n)^j(y))|\geq |f^l(f^j(y))-f^l(f_n)^j(y))|- |f^l(f^j(y))-(f_n)^l(f_n)^j(y))|>$$
$$>3\varepsilon-\varepsilon=2\varepsilon,$$
therefore $(f_n)^j(y)=x$ and $f^{-j}(Fix(f))\subset (f_n)^{-j}(Fix(f)).$\\ \\
Let $n\geq n_0$, $x\in Fix(f)$ and $w\in I$ such that $(f_n)^j(w)=x$ and $f^j(w)\neq x$. From inequality (\ref{CC*}) we get that $|(f_n)^j(w)-f^j(w)|=|x-f^j(w)|<\varepsilon$, using (\ref{okoli}) there is integer $l$ such that $|(f_n)^{j+l}(w)-f^{j+l}(w)|=|x-f^{j+l}(w)|>3\varepsilon$, hence $(f_n)^{-j}(Fix(f))\subset f^{-j}(Fix(f))$, which completes the proof.
\end{proof}
Remark. The lemma \ref{fix} can be further improved, one can show equality of the sets $\{x;f^j(x)=p \wedge f^{j-1}(x)\neq p\}=\{x;(f_n)^j(x)=p \wedge  (f_n)^{j-1}(x)\neq p\}$, where $ p $ is a fixed point of the function $ f $.
\begin{lemma}\label{rovnost}
Let f, $f_n \in C(I)$, $n\in \mathbb{N}$ be surjective maps such that f is uniform limit of $f_{1,\infty}$. If f is transitive, piecewise monotone map, and $f_{1,\infty}$ satisfies condition (CC*), then for every $x\in Prefix(f)$ there is $n_0$ such that for every $n\geq n_0$ $f_n(x)=f(x)$.
\end{lemma}
\begin{proof}
Let $x\in Prefix(f)$, then there is $j\in\mathbb{N}$ and $p\in Fix(f)$ such that $f^j(x)=p$. By Lemma \ref{fix}, there is $n_1\in\mathbb{N}$ such that for every $n\geq n_1$ we get $f_n^j(x)=p$ and $\mathrm{Orb^+}_{f_n}(x)\subset f^{-j}(p)$. Denote by $\varepsilon=\min\{|x-y|;x,y\in f^{-j}(p)\}$ then there is $n_2\in\mathbb{N}$ such that for every $n\geq n_2$, every $k\in\mathbb{N}$ and every $x\in I$ $d((f_n)^k(x),f^k(x))<\varepsilon$. Set $n_0=\min \{n_1,n_2\}$ and the result follows.
\end{proof}
\begin{proof}[Proof of Theorem \ref{main1}]
Let $0=c_1<c_2<\dots<c_N=1$ be critical points of $f$. Denote by $C=\{c_i\}_{i=1}^N$ the set of critical points of $f$. We will prove equality of $f$ and $f_n$ individually on each open interval $(c_i,c_{i+1})$.\\
Fix $(c_i,c_{i+1})$, without loss of generality we can assume that $f|_{(c_i,c_{i+1})}$ is increasing. Let $\varepsilon_{a_i}=\min\{y-c_i;y\in(c_i,c_{i+1})\wedge f(y)\in C\}$ and \\
$\varepsilon_{b_i}=\min\{c_{i+1}-y;y\in(c_i,c_{i+1})\wedge f(y)\in C\}$ and set $\varepsilon_{c_i}=\min \{\varepsilon_{a_i},\varepsilon_{b_i}\}$. Now let $p\in Fix(f)$, since $f$ is transitive the set $\bigcup_{l=0}^{\infty}f^{-l}(p)$ is dense, therefore there exists $a_i,b_i\in (c_i,c_{i+1})$ preimages of $p$ such that $a_i-c_i<\varepsilon_{c_i}/2$ and $c_{i+1}-b_i<\varepsilon_{c_i}/2$. By Lemma \ref{rovnost} there is $\varepsilon_{a_ib_i}$ and $n_{a_ib_i}$ such that for every $n\geq  n_{a_ib_i}$ is true that $f(a_i)=f_n(a_i)$, $f(b_i)=f_n(b_i)$ and for every $k\in\mathbb{N}$ $|f^k-(f_n)^k|<\varepsilon_{a_ib_i}$. Set $\varepsilon_i=\min\{\varepsilon_{c_i},\varepsilon_{a_ib_i}\}$ then there is $n_i$ such that for every $n\geq n_i$ we get $|f^k-(f_n)^k|<\varepsilon_{i}$.\\
Let $n\geq n_i$, assume that there is an integer $k$ such that for any integer $l<k$ $f(x)=f_n(x),$ $x\in f^{-l}(p)$ and there is a point $x_0\in f^{-k}(p)$ such that $f(x_0)\neq f_n(x_0)$, otherwise $f$ and $f_n$ are equal on dense set $\bigcup_{l=0}^{\infty}f^{-l}(p)$. If $x_0\in (c_i,c_{i+1})$ then there are three cases.\\
\textit{Case I}: $x_0\in [a_i,b_i]$\\
If $x_0=a_i$, resp. $x_0=b_i$ then we get contradiction with definition of $a_i,$ resp. $b_i$ and $n_i$. Therefore we can assume that $x_0\in(a_i,b_i)$. Recall that we assumed that $f|_{(c_i,c_{i+1})}$ is increasing. \\
If $f(x_0)<f_n(x_0)$, then since $f_n$ is continuous and $f_n(a_i)=f(a_i)$ there is a point $x_1\in (a_1,x_0)$ such that $f_n(x_1)=f(x_0)$ therefore $x_1\in f^{-k}(p)$. It easy to see that $f(x_1)<f_n(x_1)$ hence there is a point $x_2\in(a_i,x_1)$ such that $f_n(x_2)=f(x_1)$ and $x_2\in f^{-k}(p)$. We can continue this construction for every $m\in\mathbb{N}$, therefore $\#f^{-k}(p)=\infty$ which is contradiction with assumption that $f$ is piecewise linear. The case $f(x_0)>f_n(x_0)$ can be proved analogously.\\
\textit{Case II} $x\in(c_i,a_i)$\\
If $f(x_0)>f_n(x)$, then by analogous argumentation as in the Case I we get a contradiction.\\
If $f(x_0)<f_n(x_0)$, then by assumptions $f^2(x_0)=f(f_n(x_0))=f_n(f(x_0))=(f_n)^2(x_0)$ hence there is at least one critical point of $f$ in the interval $(f(x_0),f_n(x_0))$. Therefore one can easily show that $|f(x_0)-f_n(x_0)|\geq 2k\delta$, where $k$ is a slope of $f$ and $\delta=\min\{y-x_0;x_0<y \wedge f(y)\in C\}$, so from the definition of $a_i$ we get that
$$|f(x_0)-f_n(x_0)|>2k\frac{\varepsilon_{c_i}}{2}=k\varepsilon_{c_i}>\varepsilon_i,$$
which is a contradiction.\\
\textit{Case III:} $x_0\in(b_i,c_{i+1})$\\
The proof is analogous to the Case II.
\smallbreak
Therefore for every $n\geq n_i$ and every $x\in (c_i,c_{i+1})$ we get $f_n(x)=f(x)$. Similar result can be obtain for every interval $(c_j,c_{j+1})$, $j=1,\dots,N-1$ with respect to $n_j$. Set $n_0=\max\{n_1,n_2,\dots,n_{N-1}\}$, then for every $n\geq n_0$ and every $x\in I\setminus C$ we get $f_n(x)=f(x)$ and since $C$ is finite set the result follows.
\end{proof}
\begin{theorem}
Let f, $f_n \in C(I)$, $n\in \mathbb{N}$ be surjective maps such that f is uniform limit of $f_{1,\infty}$. If f is transitive, piecewise monotone map and $f_{1,\infty}$ satisfies condition (CC*), then there is $n_0\in\mathbb{N}$ such that for every $n\geq n_0$ $f_n(x)=f(x)$.
\end{theorem}
\begin{proof}

By \cite{Pa} every transitive piecewise monotone interval map $f$ is positively conjugate by a map $h$ to a map of constant slope $g$. Let $g_n=h\circ f_n\circ h^{-1}$. First we show that the NDS $g_{1,\infty}$ with uniform limit $g$ satisfies (CC*) condition. Fix $\varepsilon>0$ and observe that 
\begin{equation}\label{konjugacy1}
(g_n)^k=(h\circ f_n\circ h^{-1})^k=h\circ (f_n)^k\circ h^{-1}
\end{equation}

Since $h$ is a continuous map there is $\delta >0$ such that for any $x,y\in I$ such that $d(x,y)<\delta$ is true that 
\begin{equation}\label{konjugacy2}
d(h(x),h(y))<\varepsilon.
\end{equation} 
Also since $f_{1,\infty}$ satisfies (CC*) condition there is $n_0$ such that for any $n\geq n_0$, $k\in\mathbb{N}$ and $x\in I$ 
\begin{equation}\label{konjugacy3}
d((f_n)^k(x),f^k(x))<\delta
\end{equation}

Therefore for any $n\geq n_0$ by (\ref{konjugacy1}), (\ref{konjugacy2}) and (\ref{konjugacy3}) we get
\begin{align*}
d((g_n)^k(x),g^k(x))&=d((h\circ (f_n)^k\circ h^{-1})(x),(h\circ f^k\circ h^{-1})(x))=\\
&=d(h\circ ((f_n)^k\circ h^{-1})(x),h\circ (f^k\circ h^{-1})(x))<\varepsilon\\
\end{align*}
Thus $g_{1,\infty}$ satisfies (CC*) condition and the rest of the proof follows from Theorem~\ref{main1}.
\end{proof}

Theorem \ref{main1} can not be improved for general interval map, as it is showed in the following lemma. But weaker result is obtained in Theorem \ref{main2}.
\begin{lemma}
There are a transitive function $f\in C([0,1])$ and a nonautonomous system $f_{1,\infty}=\{f_i\}_{i=0}^{\infty}$, such that $f_{1,\infty}$ satisfies condition $CC^*$ and $f\neq f_n$ for every natural number $n$. 
\end{lemma}
\begin{proof}
Let $n$ be a natural number, we define $f$ as follows
$$f(x)=\left\{\begin{array}{l@{\quad}l}
0&,x=0\\
\smallbreak
2x& ,x\in \left[\frac{2^{2n+2}-4}{2^{2n+3}},\frac{2^{2n+2}-3}{2^{2n+3}}\right)\\
\smallbreak
-2x+\frac{2^{2n+2}-3}{2^{2n+1}}&,x\in  \left[\frac{2^{2n+2}-3}{2^{2n+3}},\frac{2^{2n+2}-2}{2^{2n+3}}\right)\\
\smallbreak
1&,x=\frac{1}{2}\\
\smallbreak
2x-\frac{2^{2n-2}}{2^{2n-1}}&,x\in \left[\frac{2^{2n+2}-5}{2^{2n+2}},\frac{2^{2n+2}-4}{2^{2n+2}}\right)\\
\smallbreak
2-2x&,x\in \left[\frac{2^{2n+2}-4}{2^{2n+2}},\frac{2^{2n+2}-3}{2^{2n+2}}\right)\\
\smallbreak
0&,x=1\\
linear&,elsewhere
\end{array}
\right.$$
Now, for every natural number $m$ we define $f_m$ in the following way 
$$f_m(x)=\left\{\begin{array}{l@{\quad}l}
\smallbreak
-2x+\frac{2^{2n+2}-4}{2^{2n+1}}&,x\in \left[\frac{2^{2n+2}-4}{2^{2n+3}},\frac{2^{2n+2}-3}{2^{2n+3}}\right)\quad n\geq m\\
\smallbreak
2x-\frac{1}{2^{2n+1}}&, x\in \left[\frac{2^{2n+2}-3}{2^{2n+3}},\frac{2^{2n+2}-2}{2^{2n+3}}\right)\quad n\geq m\\
\smallbreak
f(x)&, elsewhere
\end{array}
\right.$$
First we verify that the nonautonomous system $f_{1,\infty}$ satisfies condition $CC^*$,i.e. for every $\varepsilon >0$ there is an $n_0$ such that for every $n\geq 0$, $k\in\mathbb{N}$ and every  $x\in I$ 
$$d((f_n)^k(x),f^k(x))<\varepsilon$$
From the definitions of $f$ and $f_m$ it is obvious that it is sufficient to focus only on points which lies in the interval $\left[\frac{2^{2n+2}-4}{2^{2n+3}},\frac{2^{2n+2}-2}{2^{2n+3}}\right)$, but how can one easily verify, for every such $x$ it is true that $f^2(x)=(f_m)^2(x)$. Therefore we get
$$d((f_n)^k(x),f^k(x))=d\left(f\left(\frac{2^{2n+2}-3}{2^{2n+3}}\right),f_m\left(\frac{2^{2n+2}-3}{2^{2n+3}}\right)\right)=\frac{1}{2^{2n+1}}$$
Thus $f_{1,\infty}$ satisfies $CC^*$ condition.\\ \\
It remains to prove that $f$ is transitive. Let $U$ be a subset of $I$, then since the absolute value of slope of $f$ is everywhere atleast 2, there are natural numbers $k,l$ such that $y_l\in f^k(U)$, where $y_l=\frac{2^{2l+2}-4}{2^{2l+2}}$. And since $f(y_l)=\frac{1}{2^{2l-1}}$, we get that $0\in f^{k+2l+1}(U)$, therefore there is a natural number  $K$ such that $f^K(U)=[0,1]$ which completes the proof.
\end{proof}

\begin{theorem}\label{main2}
Let f, $f_n \in C(I)$, $n\in \mathbb{N}$ be surjective maps such that f is uniform limit of $f_{1,\infty}$. If f is transitive map with infinite intervals of constant slope $k$ and $f_{1,\infty}$ satisfies condition (CC*), then for every $\varepsilon>0$ there is a set $A$ and $n_0\in\mathbb{N}$ such that for every $n\geq n_0$ and every $x\in A$ we get $f_n(x)=f(x)$. Moreover $A=\bigcup_{i=1}^N A_i$, where $A_i$ is a subset of interval of monotonicity of $f$ and $\sum_{i=1}^N|A_i|\geq 1-\varepsilon$.
\end{theorem}
\begin{proof}
First observe that there is only finite set of points $\{x_i\}_{i=0}^K$ such that every neighbourhood $U_{x_i}$ of a point $x_i$  intersect infinitely many intervals of monotonicity. Hence there is a fixed point $\mathbf{x}$ such that $\#f^{-j}(\mathbf{x})<\infty$ for $j=0,1,2,\dots$\\
Analogously to Lemma \ref{fix} there is $n_0$ such that for every $n>n_0$ and every $j=0,1,2,\dots$ is true that $(f_n)^{-j}(\mathbf{x})=f^{-j}(\mathbf{x})$. Also, analogously to Lemma \ref{rovnost} for every finite set $\{y_i\}_{i=1}^L \subset \bigcup_{j=0}^{\infty}f^{-j}(\mathbf{x})$ there is $n_1$ such that for every $n>n_1$ $f_n(y_i)=f(y_i)$. Let $J$ be a interval of monotonicity of $f$ and assume that $[y_i,y_{i+1}]\subset J$ then for every $n>n_1$ and every $x\in[y_i;y_{i+1}]$ we get $f_n(x)=f(x)$. The proof of last equality is analogous to Theorem \ref{main1}, case I. 
\end{proof}
Remark. There is an open problem whether in the general interval NDS the condition (CC*) is sufficient to ensure that there is $n_0$ such that for every $n\geq n_0$ $f_n$ is topologicaly transitive if and only if $f$ is topologicaly transitive. 
\smallbreak

\noindent\textbf{Acknowledgements}

\noindent The research was supported by grant SGS/18/2019 from the Silesian University in Opava. Support of this institution is gratefully acknowledged. The authors thank Professor Marta \v{S}tef\'{a}nkov\'{a} for valuable suggestions and comments.


\begin{thebibliography}{00}

\bibitem {C1} J. S. C\'anovas. {\it Li-Yorke chaos in a class of nonautonomous discrete systems},
Journal of Difference Equation and Applications {\bf 17} (2011), 479--486

\bibitem {C2} J. S. C\'anovas. {\it On $\omega$-limit sets of of non-autonomous discrete systems},
Journal of Difference Equation and Applications {\bf 12} (2006), 95--100

\bibitem {Dv} J. Dvo\v r\'akov\'a. {\it Chaos in nonatuonomous discrete dynamical systems}, Commun Nonlinear Sci Numer Simulat {\bf 17} (2012), 4649--4652. 

\bibitem{FD} A. Fedeli and A. Le Donne. {\it A note on the uniform limit of transitive dynamical systems},
Bull. Belg. Math. Soc. Simon Stevin {\bf 16} (2009), 59--66.

\bibitem{KS} S. Kolyada, L. Snoha, \textit{Some aspects of topological transitivity-A survey}, Grazer Mathematische Berichte, 334 (1997), 3--35

\bibitem{Li} R. Li. {\it A note on stronger forms of sensitivity for dynamical systems}, 
Chaos Solitons Fractals {\bf 45} (2012), 753--758.

\bibitem{Pa} W. Parry. {\it Symbolic dynamics and transformation of the unit interval},
Trans. Amer. Math. Soc. {\bf 122} (1966), 368--378.

\bibitem{RF} H. Rom\' an-Flores. {\it Uniform convergence and transitivity},
Chaos Solitons Fractals {\bf 38} (2008), 148--153.

\bibitem{RU} S. Ruette. {\it Chaos on interval},
University Lecture Series, vol. 67, American Mathematical Society, Providence, RI, 2017.

\bibitem{SR} P. Sharma and M. Raghav. {\it On dynamics generated by a uniformly convergent sequence of maps}, 
TOPOL APPL {\bf 247} (2018), 81--90.

\bibitem{St} M. \v Stef\'ankov\'a. {\it Inheriting of chaos in uniformly convergent nonautonomous dynamical systems of the interval}, Discrete and Continuous Dynamical Systems {\bf 36} (2016), 3435-3443.

\bibitem{YZZ} K. Yanm F. Zeng and G. Zhang. {\it Devaney's chaos on uniform limit maps},
Chaos Solitons Fractals {\bf 44} (2011), 522--525.
\end{thebibliography}
\end{document}